\documentclass{amsart}
\usepackage{thmtools}
\usepackage{thm-restate}
\usepackage{amssymb}
\usepackage{mathtools}
\usepackage{float}
\usepackage{hyperref}
\newcommand{\Z}{\mathbb Z}
\newcommand{\sm}[4]{\begin{psmallmatrix}#1&#2\\#3&#4 \end{psmallmatrix}}
\newcommand{\set}[1]{\left\{#1\right\}}
\newcommand{\inv}{^{-1}}
\newcommand{\paren}[1]{\left(#1\right)}
\newcommand{\ceil}[1]{\left\lceil#1\right\rceil}

\newcommand{\ZZ}{\mathbb{Z}}

\newtheorem{theorem}{Theorem}
\newtheorem{lemma}[theorem]{Lemma}

\numberwithin{equation}{section}

\begin{document}

\title[Coefficients of modular functions]{Congruences for coefficients of modular functions in levels 3, 5, and 7 with poles at 0}

\author{Paul Jenkins}
\address{Department of Mathematics, Brigham Young University, Provo, UT 84602}
\curraddr{}
\email{jenkins@math.byu.edu}
\thanks{This work was partially supported by a grant from the Simons Foundation ($\# 281876$ to Paul Jenkins).}

\author{Ryan Keck}
\address{}
\curraddr{}
\email{ryank@mathematics.byu.edu}
\thanks{}


\date{\today}

\dedicatory{}

\commby{}

\begin{abstract}
We give congruences modulo powers of $p \in \{3, 5,7\}$ for the Fourier coefficients of certain modular functions in level $p$ with poles only at 0, answering a question posed by Andersen and the first author and continuing work done by the authors and Moss. The congruences involve a modulus that depends on the base $p$ expansion of the modular form's order of vanishing at $\infty$.
\end{abstract}

\maketitle

\section{Introduction}
A modular form $f(z)$ of level $N$ and weight $k$ is a complex valued function which is holomorphic on the upper half plane, satisfies the equation
\[f\paren{\frac{az+b}{cz+d}} = (cz+d)^k f(z) \text{ for all } \sm{a}{b}{c}{d} \in \Gamma_0(N),\]
and is holomorphic at the cusps of $\Gamma_0(N).$ Letting $q=e^{2 \pi i z}$, modular forms have a Fourier expansion $f(z) = \sum_{n \geq 0} a(n)q^n$ with Fourier coefficients $a(n)$. A {\em weakly holomorphic} modular form is a modular form that is allowed to be meromorphic at the cusps; we define $M_k^\sharp (N)$ to be the space of weakly holomorphic modular forms of weight $k$ and level $N$ that are holomorphic away from the cusp at $\infty,$ and in the same notation as~\cite{jenkins}, we use $M_k^\flat(N)$ to denote forms holomorphic away from the cusp at 0. For prime $N$, these are the only cusps. Modular forms in both of these spaces also have Fourier expansions at infinity, where the constraint $n \geq 0$ is relaxed to $n \gg -\infty$ if the form has a pole at infinity.

The Fourier coefficients of modular forms often satisfy interesting congruences. For the $j$-invariant $j(z) = q \inv + 744 + \sum_{n=1}^\infty c(n)q^n \in M_0^\sharp(1)$, Lehner~\cite{lehner1,lehner2} proved that the $c(n)$ satisfy the congruence
\[c(2^a3^b5^c7^dn) \equiv 0 \pmod{2^{3a+8}3^{2b+3}5^{c+1}7^d} \text{ if } a, b, c, d \geq 1.\]
Kolberg~\cite{kolberg1,kolberg2}, Aas~\cite{aas}, and Allatt and Slater~\cite{allatt} strengthened Lehner's congruences for $j(z)$. Furthermore, Griffin~\cite{griffin} gave a canonical basis for $M_0^\sharp(1)$ and extended Kolberg's and Aas's results to the basis elements. Similarly, the first author, Andersen, and Thornton~\cite{andersen,thornton1} proved congruences for the Fourier coefficients of elements of canonical bases for $M_0^\sharp(p)$ with $p = 2, 3, 5, 7$. The authors and Moss~\cite{jenkins} proved congruences for the Fourier coefficients of elements of a canonical basis for $M_0^\flat(2)$. It is natural to wonder whether a similar result holds for bases of $M_0^\flat(p)$ for the other genus zero primes, mirroring the results of $M_0^\sharp(p)$.

Let $p \in \{2,3,5,7,13\}$. Taking $\eta(z) = q^{\frac{1}{24}} \prod_{n=1}^\infty \left(1 - q^n \right)$ to be the Dedekind eta function, a Hauptmodul for $\Gamma_0(p)$ is
\[\phi^{(p)}(z) = \paren{\frac{\eta(pz)}{\eta(z)}}^{24/(p - 1)} =  q+\frac{24}{p - 1}q^2+\cdots,\]
which vanishes at $\infty$ and has a pole only at 0. The functions $(\phi^{(p)}(z))^m$ for $m \geq 0$ are a basis for $M_0^\flat(p).$ Andersen and the first author used powers of $\phi^{(p)}(z)$ to prove congruences involving $ \psi^{(p)} = \frac{1}{\phi^{(p)}} = q\inv - \frac{24}{p - 1} + \cdots \in M_0^\sharp(p)$ in~\cite{andersen}, and made the following remark: ``Additionally, it appears that powers of the function $[\phi^{(p)}(z)]$ have Fourier coefficients with slightly weaker divisibility properties... It would be interesting to more fully understand these congruences.'' The authors with Moss in~\cite{jenkins} proved congruences for the Fourier coefficients of $\phi^{(2)}(z)$ and its powers. In this paper, we use similar techniques to obtain congruences for $\phi^{(p)}(z)$ and its powers for $p = 3, 5, 7$.

Write $\phi^{(p)}(z)^m = \sum_{n=m}^\infty a^{(p)}(m,n)q^n$. Let $\chi_S$ be the characteristic function on $S$, which outputs $1$ when the input is an element of $S$ and $0$ otherwise. The main result of this paper is the following theorem.

\begin{restatable}{theorem}{congruence}
\label{congruence}
Let $p \in \{3, 5, 7\}$. Let $n = p^\alpha n'$ where $p \nmid n'.$ Express the base $p$ expansion of $m$ as $a=\sum_{i=1}^{\infty}a_i p^{i - 1}$, where $a_i=0$ for all sufficiently large $i$. Consider the rightmost $\alpha$ digits $a_\alpha\dots a_2 a_1$. Let $i'$ be the index of the rightmost nonzero digit, or $i' = -1$ if $a_1 = a_2 = \cdots =  a_\alpha = 0$. Let
\begin{align*}
    \gamma_3(m,\alpha) &=
    \begin{cases}
    3 - a_{i'} + 2\#\set{i\ | \ a_i = 0, i > i'} + \#\set{i\ | \ a_i = 1, i > i'} &\text{if } i' \geq 0,\\
    0 &\text{otherwise.}
    \end{cases}\\
    \gamma_5(m,\alpha) &= \gamma_7(m,\alpha)=
    \begin{cases}
    \chi_{\set{1, 2}}(a_{i'}) + \#\set{i\ | \ a_i \in \set{0,1}, i > i'} &\text{if } i' \geq 0,\\
    0 &\text{otherwise.}
    \end{cases}\\
\end{align*}
Then
\[a^{(p)}(m,p^\alpha n') \equiv 0 \pmod{p^{\gamma_p(m,\alpha)}}.\]
\end{restatable}
The power of $p$ in the congruence includes a count of the number of digits in the base $p$ expansion of $m$ that are $0$, $1$, or $2$. This result is similar to the one found for $\phi^{(2)}$, but it is more complicated to state because there are more digits in bases $3, 5, 7$. We note that this congruence is not sharp. For $m = 1,$ Allatt and Slater in~\cite{allatt} proved a stronger result that provides an exact congruence for many $n$.

As an example, the base $3$ expansion for $m = 102$ is $m = \cdots 00010210.$ Table~\ref{gammaexample} gives values of $\gamma_3$.
\begin{table}[b]
\label{gammaexample}
\begin{tabular}{|c||c|c|c|c|c|c|c|c|c|c|c|c|c|c|}
\hline
$\alpha$            & 0 & 1 & 2 & 3 & 4 & 5 & 6 & 7 & 8 & 9 & $\cdots$ & $\alpha$     & $\cdots$\\
\hline
$\gamma(102,\alpha)$ & 0 & 0 & 2 & 2 & 4 & 5 & 7 & 9 & 11 & 13 & $\cdots$ & $2(\alpha - 5)$ + 5 & $\cdots$\\ \hline
\end{tabular}
\caption{Values of $\gamma_3(m,\alpha)$ for $m = 102$}
\end{table}
Notice that once $\alpha$ surpasses 5---the leftmost nonzero digit in the base $3$ expansion of $m$ occurs in the 5th place---$\gamma_3$ always increases by 2 as $\alpha$ increases by 1. This illustrates that $\gamma_3(m,\alpha)$ is unbounded for a fixed $m$. Similar examples can be constructed for $\gamma_5$ and $\gamma_7$.

Section 2 contains the machinery and definitions we use for proving Theorem~\ref{congruence} and the proof itself is in Section 3. In Section $4$ we discuss the $p = 13$ case; although $13$ is also a genus zero prime, we do not obtain congruences modulo $13$.

We thank the referee for helpful comments which helped to improve this paper.

\section{Preliminary Lemmas}

The operator $U_p$ on a function $f(z)$ is given by
\[U_p f(z) = \dfrac{1}{p} \sum_{j = 0}^{p-1} f\paren{\frac{z+j}{p}}.\]
If $M_k^!(N)$ is the space of weakly holomorphic modular forms of weight $k$ and level $N$, meaning we allow poles at any cusp, then we have $U_p : M_k^!(N) \to M_k^!(N)$ if $p$ divides $N$. If $f(z)$ has the Fourier expansion $\sum^\infty_{n = n_0}a(n)q^n$, then the effect of $U_p$ on the Fourier expansion is given by $U_p f(z) = \sum^\infty_{n = n_0}a(pn)q^n$.

We prove Theorem~\ref{congruence} by first proving several facts about $\phi^{(p)}$ and its image under the $U_p$ operator.  Most importantly, $U_p(\phi^{(p)})$ is a polynomial with integer coefficients in $\phi^{(p)}$.  If all of the coefficients are divisible by some number, then every $p$th coefficient of $\phi^{(p)}$ is divisible by that number. The same holds true when applying $U_p$ to any power of $\phi^{(p)}$ as well.

The following result describes how $U_p$ applied to a modular function behaves under the Fricke involution. This will help us in Lemma~\ref{polynomial} to write $U_p(\phi^{(p)})^m$ as a polynomial in $\phi^{(p)}$.

\begin{lemma}[{\cite[Theorem~4.6]{apostol}}]
\label{fricke}
Let $p$ be prime and let $f(z)$ be a level $p$ modular function. Then
\[p(U_pf)\paren{\frac{-1}{pz}} = p(U_pf)(pz)+f\paren{\frac{-1}{p^2z}}-f(z).\]
\end{lemma}

The Fricke involution $\sm{0}{-1}{p}{0}$ swaps the cusps of $\Gamma_0(p)$, which are 0 and $\infty.$ We will use this fact in the proof of Lemma~\ref{polynomial}, and the following relations between $\phi^{(p)}(z)$ and $\psi^{(p)}(z)$ will help us compute this involution.

\begin{lemma}[{\cite[Lemma~3]{andersen}}]
\label{phipsi}
The functions $\phi(z)$ and $\psi(z)$ satisfy the relations
\begin{align*}
\phi^{(p)}\paren{\frac{-1}{pz}} &= p^{-12/(p - 1)}\psi^{(p)}(z),\\
\psi^{(p)}\paren{\frac{-1}{pz}} &= p^{12/(p - 1)}\phi^{(p)}(z).
\end{align*}
\end{lemma}

The following lemma is a special case of a result of Lehner~\cite{lehner2}. It provides a polynomial with functions as its coefficients whose roots are modular forms used in the proof of Theorem~\ref{betterbound}.

\begin{lemma}[{\cite[Theorem~2]{lehner2}}]
\label{polyrelation}
There exist integers $b_j^{(p)}$ such that
\[U_p\phi^{(p)}(z) = p \sum_{j = 1}^p b_j^{(p)} \phi^{(p)}(z)^j.\]
Furthermore, let $h^{(p)}(z) = p^{12/(p - 1)}\phi^{(p)}(z/p)$. Then
\[(h^{(p)}(z))^p + \sum_{j = 1}^p (-1)^j g_j(z)(h^{(p)}(z))^{p - j} = 0\]
where
\[g_j(z) = (-1)^{j+1} p^{12/(p-1) + 2} \sum_{\ell = j}^p b_\ell^{(p)} \phi^{(p)}(z)^{\ell - j + 1}.\]
\end{lemma}

In the following lemma, we extend the result from the first part of Lemma~\ref{polyrelation}, writing $U_p(\phi^{(p)})^m$ as an integer polynomial in $\phi^{(p)}$. In particular, we give the degree of that polynomial. An alternative approach can be seen in~\cite[Lemma 4.1.1]{calegari}.

\begin{lemma}
\label{polynomial}
For all $m \geq 1$, $U_p (\phi^{(p)})^m \in \Z[\phi^{(p)}].$ In particular,
\[U_p(\phi^{(p)})^m = \sum_{j = \ceil{m/p}}^{pm} d(m,j)(\phi^{(p)})^j\] where $d(m,j) \in \Z$, and  $d(m,pm)$ is not 0.
\end{lemma}

\begin{proof}
We proceed as in~\cite[Lemma 5]{jenkins}; this is a straightforward generalization from $2$ to $p$. Using Lemmas~\ref{fricke} and~\ref{phipsi}, we have that

\begin{align*}
U_p\phi^{(p)}(-1/pz)^m
&= U_p\phi^{(p)}(pz)^m + p\inv\phi^{(p)}(-1/p^2z)^m-p\inv \phi^{(p)}(z)^m \\
&= U_p\phi^{(p)}(pz)^m + p^{-1-12m/(p - 1)} \psi^{(p)}(pz)^m - p\inv \phi^{(p)}(z)^m\\
&= p^{-1-12m/(p - 1)}q^{-pm} + O(q^{-pm+p}).
\end{align*}
Thus,
\begin{align*}
p^{1+12m/(p - 1)}U_p\phi^{(p)}(-1/pz)^m &= q^{-pm} + O(q^{-pm+p}).
\end{align*}
Because $\phi^{(p)}(z)^m$ is holomorphic at $\infty,$ $U_p \phi^{(p)}(z)^m$ is holomorphic at $\infty$. So $U_p \phi(-1/pz)^m$ is holomorphic at 0 and, since the Fourier expansion starts with $q^{-pm}$, it must be a polynomial of degree $pm$ in $\psi^{(p)}(z).$ Let $b(m,j) \in \Z$ such that
\[
p^{1+12m/(p - 1)}U_p\phi^{(p)}(-1/pz)^m = \sum_{j = 0}^{pm} b(m,j)\psi^{(p)}(z)^j,
\]
and we note that $b(m, pm)$ is not 0. Now replace $z$ with $-1/pz$ and use Lemma~\ref{phipsi} to get
\[
p^{1+12m/(p - 1)}U_p\phi^{(p)}(z)^m = \sum_{j = 0}^{pm} b(m,j)p^{12j/(p - 1)}\phi^{(p)}(z)^j,
\]
which gives
\[
U_p\left(\phi^{(p)}(z)^m\right) = \sum_{j = 0}^{pm} b(m,j)p^{12(j-m)/(p - 1) - 1}\phi^{(p)}(z)^j.
\]
Because $\left(\phi^{(p)}(z)\right)^m = q^m + \cdots$, if $m$ is divisible by $p$, the leading term of the above sum is $q^{m/p}$, and otherwise the smallest power of $q$ present in the polynomial is at least $\ceil{m/p}$, so the sum starts with $j = \ceil{m/p}$ as desired. Notice that $b(m,j)p^{12(j-m)/(p - 1)-1}$ is an integer because the coefficients of $\phi^{(p)}(z)^m$ are integers.
\end{proof}

We may repeatedly use Lemma~\ref{polynomial} to write $U_p^\alpha (\phi^{(p)})^m$ as a polynomial in $\phi^{(p)}$. Let
\begin{equation}
\label{fdefn}
f_{(p)}(\ell) = \ceil{\ell/p},\ f^0_{(p)}(\ell) = \ell, \text{ and } f^\alpha_{(p)}(\ell) = f_{(p)}(f^{\alpha-1}_{(p)}(\ell)) \qquad \text{for } \alpha \geq 1.
\end{equation}
Using Lemma~\ref{polynomial}, the smallest exponent of $q$ appearing in $U_p^\alpha (\phi^{(p)})^m$ is at least $f_{(p)}^\alpha(m).$

Lemma~\ref{binarygammma} provides a connection between $\gamma_p(m,\alpha)$ and the integers $f_{(p)}^\alpha(m)$: $\gamma_p$ is counting the number of $0$s and $1$s to the left of the first nonzero digit in the base $p$ expansion of $m$. The key difference between the following lemma and its corresponding lemma in~\cite{jenkins} is that there are more digits in bases $3, 5, 7$.
\begin{lemma}
\label{binarygammma}
The number of $0$s to the left of the rightmost nonzero digit in the first $\alpha$ digits of the base $p$ expansion of $m$ is equal to the number of integers congruent to $1$ modulo $p$ in the list
\[m, f_{(p)}(m), f_{(p)}^2(m), \dots, f_{(p)}^{\alpha -1}(m), \]
except when the rightmost nonzero digit is $1$ (in which case there is exactly one more in the list). Similarly, the number of $1$s to the left of the rightmost nonzero digit in the base $p$ expansion of $m$ is equal to the number of integers congruent to $2$ modulo $p$ in the above list, again with the exception of when the rightmost nonzero digit is $2$.
\end{lemma}

\begin{proof}
Write the base $p$ expansion of $m$ as $a_r\dots a_2 a_1$, and consider its first $\alpha$ digits, $a_\alpha \dots a_2 a_1$, where $a_i = 0$ for $i > r$ if $\alpha > r.$ If $a_i = 0$ for $1 \leq i \leq \alpha$, then all of the integers in the list are zero modulo $p$. Otherwise, suppose that $a_i = 0$ for $1 \leq i < i'$ and $a_{i'} \neq 0.$ Apply $f_{(p)}$ repeatedly to $m$. Each application of $f_{(p)}$ deletes the rightmost $0$ from the expansion, until $a_{i'}$ is the rightmost remaining digit; that is, $f_{(p)}^{i'-1}(m) = a_\alpha \dots a_{i'-1}a_{i'}.$ In particular, the rightmost digit is nonzero. Having reduced to this case, we now treat only the case where $m$ is not divisible by $p$.

If $m$ is not divisible by $p$, and $a_1 \in \{1, 2\}$, then at least one number in the list, namely $m$, is congruent to either $1$ or $2$ modulo $p$. Also, $f_{(p)}(m) = \ceil{m/p} = (m+a)/p$ for some $a \neq 0$. Applied to the base $p$ expansion of $m$, $f_{(p)}$ deletes $a_1$ and propagates a $1$ leftward through the base $p$ expansion, replacing any digit that was previously equal to $p-1$ with zero. This is essentially the operation of carrying in addition. This process then terminates upon encountering the rightmost digit less than $p - 1$ (if it exists), which becomes one greater. As in the case where $p$ divides $m$, we apply $f$ repeatedly to delete the new leading 0s. But if the first nonzero digit to the left was either a $0$ or a $1$ before we propagated a $1$ leftward, it is now a $1$ or a $2$ respectively. So now when we repeat this process until all digits are accounted for, we notice that any digit that was either a $0$ or a $1$ becomes a $1$ or a $2$ respectively (with the exception of the first nonzero digit), which proves the lemma.
\end{proof}

\section{Proof of the Main Theorem}

Theorem~\ref{congruence} will follow from the next theorem.

\begin{restatable}{theorem}{betterbound}
\label{betterbound}
Let $f_{(p)}^\alpha(m)$ be as in~(\ref{fdefn}). Let $\gamma_p(m,\alpha)$ be as in Theorem~\ref{congruence}, and let $\alpha \geq 1$. Define $P^{(p)}(\ell, a)$ to be the set of polynomials in $\phi^{(p)}$ with lowest power $\ell$ having coefficient $d_\ell$ divisible by $p^a$ and each subsequent coefficient $d_k$ being divisible by at least $p^{\delta_p(k - \ell) + a}$, where $\delta_3 = 4$ and $\delta_5 = \delta_7 = 1$. Then
\begin{equation}
\label{goodresult}
U_p^\alpha (\phi^{(p)})^m \in P^{(p)}(f_{(p)}^{\alpha}(m), \gamma_p(m, \alpha)).
\end{equation}
\end{restatable}
Since every coefficient of such a polynomial is divisible by $p^{\gamma_p(m, \alpha)}$, this shows that the coefficient of $q^{p^\alpha n}$ in $(\phi^{(p)})^m$ must vanish modulo $p^{\gamma_p(m, \alpha)}$.  Note that these methods do not give meaningful congruences for the case when $\alpha = 0$.

Theorem~\ref{betterbound} is an improvement on the following result by Lehner~\cite{lehner2} for $p = 3$.
\begin{theorem}
\label{lehnerbound}
\cite[Equation~3.24]{lehner2}
Write $U_3 ^\alpha (\phi^{(3)})^m$ as  $\sum d(m,j,\alpha)(\phi^{(3)})^j \in \Z[\phi^{(3)}].$ For any integer $k$, let $\nu_3(k)$ be the highest power of $3$ dividing $k$. Then
\begin{align*}
    \nu_3(d(m,j,\alpha)) \geq 4(j-1) + \alpha(2 - 4(1 - m)).
\end{align*}
\end{theorem}
In particular, Lehner's bound sometimes only gives the trivial result that the 3-adic valuation of $d(m,j,\alpha)$ is greater than some negative integer. Lehner also proved congruences for $p = 5, 7$, but they experience similar issues~\cite{lehner1}.

In proving Theorem~\ref{betterbound}, we use techniques from Lemmas~5 and~6 of~\cite{andersen}. From the definition of $U_p$, we have
\begin{equation}
\label{defUp}
    U_p\phi^{(p)}(z)^m = p\inv \sum_{j = 1}^p \phi^{(p)}\left(\dfrac{z + j}{p}\right)^m.
\end{equation}
Define $h^{(p)}_\ell(z) = p^{12/(p - 1)}\phi^{(p)}\paren{\frac{z+\ell}{p}}.$  Consider the polynomials
\[
F^{(p)}(x) = x^p + \sum_{j = 1}^{p - 1}(-1)^jg_j(z)x^{p - j}
\]
with the $g_j(z)$ functions from Lemma~\ref{polyrelation} as coefficients.  We claim that $F^{(p)}(x)$ has the $h^{(p)}_\ell(z)$ as roots. By Lemma~\ref{polyrelation}, we know that $h^{(p)}_0(z)$ is a root; the others are roots because the $g_j(z)$ are fixed under $z \mapsto z + 1$, but for this transformation of $z$, $h^{(p)}_\ell(z)$ gets sent to $h^{(p)}_{\ell + 1}(z)$.

Since $F^{(p)}$ has the functions $h^{(p)}_\ell(z)$ as its roots, the coefficients $g_j(z)$ are the symmetric polynomials in the roots. We can now use Newton's identities for the sums of powers of roots of a polynomial. Writing $F^{(p)}(x) = \prod_{i = 1}^n (x - x_i),$ let $S_\ell = x_1^\ell + \cdots + x_n^\ell$.  It follows that
\[
S_\ell = g_1S_{\ell-1}-g_2S_{\ell-2}+\cdots+(-1)^{\ell+1}\ell g_\ell.
\]

Let $R^{(p)}$ be the set of polynomials of the form $\sum_{n=1}^N d_n \phi^{(p)}(z)^n$, where $d_n \in \ZZ$ and where for $n \geq 2$, $\nu_p(d_n) \geq \delta_p (n-1)$, where $\delta_p$ is as in Theorem~\ref{betterbound}.  We will need the following \emph{$R^{(p)}$ product lemma}, which implies that when multiplying two polynomials in $R^{(p)}$, we must multiply in $\delta_p$ extra copies of $p$ for the product to be in $R^{(p)}$.
\begin{lemma}
($R^{(p)}$ product lemma) If $f, g \in R^{(p)}$, then $p^{\delta_p}fg \in R^{(p)}$.
\end{lemma}
\begin{proof}
We only need to prove this for the product
\begin{align*}
    p^{\delta_p}\left(p^{\delta_p(i - 1)}d_i\phi^{(p)}(z)^i\right)\left(p^{\delta_p(j - 1)}d_j'\phi^{(p)}(z)^j\right),
\end{align*}
since then the lemma will hold for the product of any two polynomials by linearity, as the sum of any two elements in $R^{(p)}$ is clearly also an element of $R^{(p)}$. Observe that
\begin{align*}
    p^{\delta_p}\left(p^{\delta_p(i - 1)}d_i\phi^{(p)}(z)^i\right)\left(p^{\delta_p(j - 1)}d_j'\phi^{(p)}(z)^j\right) &= p^{\delta_p(i - 1 + j - 1 + 1)}d_id_j'\phi^{(p)}(z)^{i + j}\\
    &= p^{\delta_p((i + j) - 1)}(d_id_j')\phi^{(p)}(z)^{i + j},
\end{align*}
which is clearly an element of $R^{(p)}$.
\end{proof}

We prove Theorem~\ref{betterbound} by first showing the theorem holds when $\alpha = 1$.  This is comparable to Lemma 6 from~\cite{andersen}, which gives a subring of $\Z[\phi]$ which is closed under the $U_p$ operator. Here, we employ a similar technique to prove divisibility properties of the polynomial coefficients in Lemma~\ref{polynomial}. We then show that applying $U_p$ to a polynomial in the set $P^{(p)}(f_{(p)}^{\alpha}(m), \gamma_p(m, \alpha))$ will carry it to the set $P^{(p)}(f_{(p)}^{\alpha + 1}(m), \gamma_p(m, \alpha + 1))$, which we refer to as the polynomial step. Implicitly, this proves the result by induction. This structure differs from~\cite{jenkins} because it allows us to prove the polynomial step in a much cleaner way. Another approach to proving the base case can be found in~\cite[Lemma 4.1.1]{calegari}.  In this paper, we will only treat the case $p = 3$, but the computations are similar for $p = 5, 7$. Therefore, we will simply use $\phi$ in place of $\phi^{(3)}$.

Let $\alpha = 1$.  We seek to prove the statement
\[U_p(\phi)^m = \sum_{j = \ceil{m/p}}^{pm} d(m, j)(\phi)^j\]
with
\begin{equation}
\label{alpha1bound}
p^{\delta_p(j - \ceil{m/p}) + c^{(p)}_m} \mid d(m, j)
\end{equation}
where
\begin{align*}
c^{(3)}_m =
\begin{cases}
2 & m \equiv 1 \pmod{3},\\
1 & m \equiv 2 \pmod{3},\\
0 & \text{otherwise,}\end{cases}\\
c^{(5)}_m = c^{(7)}_m =
\begin{cases}
1 & m \equiv 1, 2 \pmod{p},\\
0 & \text{otherwise.}
\end{cases}
\end{align*}
We prove~(\ref{alpha1bound}) by induction on $m$.

Given Lemma~\ref{polynomial} and the statement we are trying to prove, we see that proving~(\ref{alpha1bound}) is equivalent to proving the following three statements:
\begin{align*}
    m \equiv 0 \pmod{3}:&\ \exists\ r \in R^{(3)},\ U_3 \phi^m = 3^{-4(\lceil m/3 \rceil - 1)}r,\\
    m \equiv 1 \pmod{3}:&\ \exists\ r \in R^{(3)},\ U_3 \phi^m = 3^{-4(\lceil m/3 \rceil - 1) + 2}r,\\
    m \equiv 2 \pmod{3}:&\ \exists\ r \in R^{(3)},\ U_3 \phi^m = 3^{-4(\lceil m/3 \rceil - 1) + 1}r.
\end{align*}
Furthermore, by equation~(\ref{defUp}) and considering the functions $h_{\ell}^{(3)}$, we see that $U_3 \phi^m = 3^{-1 - 6m}S_m$. Let $c_m = 0,2,1$ if $m$ is congruent to $0,1,2$ modulo $3$ respectively. Then for the theorem to be true, we require
\begin{equation}
\label{SmMustBe}
    S_m = 3^{6m + 5 - 4 \lceil m/3 \rceil + c_m}r
\end{equation}
for some $r \in R^{(3)}$. We will prove this by induction.

Our base cases are $m = 1, 2, 3$. From Lemma~\ref{polyrelation},
\begin{align*}
    g_1(z) &= 10 \cdot 3^9\phi(z) + 4 \cdot 3^{14}\phi^2(z) + 3^{18}\phi^3(z),\\
    g_2(z) &= -4\cdot 3^{14}\phi(z) - 3^{18}\phi^2(z),\\
    g_3(z) &= 3^{18}\phi(z).
\end{align*}
We find $S_1, S_2,$ and $S_3$ as follows:
\begin{align*}
    S_1 &= g_1 = 3^9(Q_1(\phi)),\\
    S_2 &= g_1S_1 - 2g_2 = 3^{14}(Q_2(\phi)),\\
    S_3 &= g_1S_2 - S_1g_2 + 3g_3 = 3^{19}(Q_3(\phi)),
\end{align*}
where $Q_1, Q_2, Q_3 \in R^{(3)}$. Note that since $g_1$ and $S_1$ each are of the form $3^9 r$ for some $r \in R^{(3)}$, we use the $R^{(p)}$ product lemma to quickly deduce that their product is of the form $3^{18-4}r = 3^{14}r$ for some $r \in R^{(3)}$, from which we easily see $S_2 = 3^{14}(Q_2(\phi))$. Similarly, $3^{23}$ divides both $g_1S_2$ and $S_1g_2$, which means that to keep $g_1S_2$ and $S_1g_2$ in $R^{(p)}$, we lose $3^4$ and get $S_3 = 3^{19}(Q_3(\phi))$ using the same lemma. Comparing with~(\ref{SmMustBe}), we see that the base case is proved.

Suppose that for some $m\geq 4, i \in \{1, 2, 3\}$ we have
\begin{align*}
    S_{m - i} = r_i3^{6(m - i) + 5 - 4 \lceil (m - i)/3 \rceil + c_{m - i}}
\end{align*}
for some $r_i \in R^{(3)}$. We want to show that $S_m = 3^{6m + 5 - 4 \lceil m/3 \rceil + c_m}r_0$ for some $r_0 \in R^{(3)}$. Recall that for $m \geq 4$,
\begin{align*}
    S_m = g_1 S_{m - 1} - g_2S_{m - 2} + g_3S_{m - 3}.
\end{align*}
If each of the terms is of the form $3^{6m + 5 - 4 \lceil m/3 \rceil + c_m}r_0$ for some $r_0 \in R^{(3)}$, then we are done. We will check each term.

Using the $R^{(p)}$ product lemma, the power of $3$ that we get from $g_1 S_{m - 1}$ is
\begin{align*}
    &9 + 6(m - 1) + 5 - 4 \lceil (m - 1)/3 \rceil + c_{m - 1} - 4\\&=6m + 4 - 4 \lceil (m - 1)/3 \rceil + c_{m - 1},
\end{align*}
which we want to be greater than $6m + 5 - 4\lceil m/3 \rceil + c_m$. In other words, we want to check whether
\begin{align*}
    -1 -4\lceil (m - 1)/3 \rceil + c_{m - 1} \geq - 4\lceil m/3 \rceil + c_m,
\end{align*}
i.e.
\begin{equation}
\label{cMcheck}
    -1 + c_{m - 1} \geq - 4(\lceil m/3 \rceil - \lceil (m - 1)/3 \rceil) + c_m.
\end{equation}
Checking the three cases depending on the value of $m$ modulo $3$, we see that~(\ref{cMcheck}) is true.

The power of $3$ that divides $g_2 S_{m - 2}$ is
\begin{align*}
    &14 + 6(m - 2) + 5 - 4 \lceil (m - 2)/3 \rceil + c_{m - 2} - 4\\&= 6m + 3 - 4 \lceil (m - 2)/3 \rceil + c_{m - 2},
\end{align*}
which we want to be greater than $6m + 5 - 4\lceil m/3 \rceil + c_m$. In other words, we want to check whether
\begin{align*}
    -2 + c_{m - 2} &\geq - 4(\lceil m/3 \rceil - \lceil (m - 2)/3 \rceil) + c_m.
\end{align*}
As before, we find that this is true.

In the final term, the power of $3$ that divides $g_3 S_{m - 3}$ is
\begin{align*}
    &18 + 6(m - 3) + 5 - 4 \lceil (m - 3)/3 \rceil + c_{m - 3} - 4\\&= 6m + 1 -4 \lceil (m - 3)/3 \rceil + c_{m - 3},
\end{align*}
which we want to be greater than $6m + 5 - 4\lceil m/3 \rceil + c_m$. In other words, we want to check whether
\begin{align*}
    -4 -4\lceil (m - 3)/3 \rceil + c_{m - 3} &\geq - 4\lceil m/3 \rceil + c_m,
\end{align*}
i.e.
\begin{align*}
    -4 &\geq - 4(\lceil m/3 \rceil - \lceil (m - 3)/3 \rceil).
\end{align*}
As before, this is true, proving the theorem for $\alpha = 1$.

Moving forward to the polynomial step, we have shown that
\[
U_3(\phi(z)^m) = \sum_{n = i}^{3m} d(m, n) (\phi(z))^n,
\]
where $i = \lceil \frac{m}{3} \rceil$, and $3^2 \mid d(m, i)$ if $m \equiv 1 \pmod{3}$, or $3 \mid d(m, i)$ if $m \equiv 2 \pmod{3}$. Furthermore, $\nu_3(d(m, n)) \geq 4(n - i) + c_m$, where $n \neq i$. To take into account the equivalence class of $k$ modulo $3$, we can rewrite this as
\[
U_3(\phi(z)^m) = 3^{c_m} \sum_{n = i}^{3m} d'(m, n) (\phi(z))^n,
\]
where $c_m$ is as above. 

Suppose that
\[
\sum_{n = i}^{j} d_n (\phi(z))^n
\]
is a polynomial where $\nu_3(d_n) \geq 4(n - i)$ for $n \geq i$; without loss of generality, we assume that $3 \nmid d_i$. We will show that
\begin{equation}
U_3\left(\sum_{n = i}^{j} d_n \phi(z)^n\right) = 3^{c_i}\sum_{n = i'}^{3j} d_n' \phi(z)^n, \label{polynomialstep}
\end{equation}
where $d_n' \in \mathbb{Z}$, the quantity $3^{c_i}$ is $3^2, 3,$ or $1$ if $i$ is $1,2, $ or $0$ modulo 3 respectively, and $\nu_3(d_n') \geq 4(n - i') + c_i$ for $n \geq i'$.

We have already shown that $i'$ is at least $\lceil\frac{i}{3} \rceil$ and that the degree of the resulting polynomial is $3j$, so the only question is whether the $3^{c_i}$ term appears and whether the $d_n'$ satisfy the appropriate divisibility by increasing powers of $3$.

We use the sets $P^{(3)}(\ell, a)$ as defined in the statement of Theorem~\ref{betterbound}.  Note first that $U_3(d_i \phi^i)$ is in the set $P^{(3)}(i', c_i)$.
Since $U_3$ is linear over the sum, we find the set to which each $U_3(d_n \phi^n)$ belongs, and then show that each of those sets is contained in $P^{(3)}(i', c_i)$.  It follows that~(\ref{polynomialstep}) holds.  We do this by showing the containments
\begin{align*}
P^{(3)}\left(i', c_i \right) &\supseteq P^{(3)}\left( \left\lceil \frac{i + 1}{3} \right\rceil, c_{i + 1} + 4(i + 1 - i))\right)\\&\supseteq \cdots\\&\supseteq P^{(3)}\left( \left\lceil \frac{j}{3} \right\rceil, c_{j} + 4(j - i)\right).
\end{align*}
In order to show all of these containments, we only need to show consecutive set inclusions, or that for each $m$, $P^{(3)}(\lceil \frac{m}{3} \rceil, c_{m} + 4(m - i)) \subseteq P^{(3)}(\lceil \frac{m-1}{3} \rceil, c_{m-1} + 4(m - 1 - i))$.

Now we move to cases.\\
\textit{Case 1.} If $m \equiv 0 \pmod{3}$, then $\lceil\frac{m - 1}{3}\rceil = \lceil\frac{m}{3}\rceil$, and we want to show that $P^{(3)}(\lceil \frac{m}{3} \rceil, 4(m - i)) \subseteq P^{(3)}(\lceil \frac{m-1}{3} \rceil, 1 + 4(m - 1 - i))$. Observe that
\begin{align*}
P^{(3)}\left( \left\lceil \frac{m}{3} \right\rceil, 4(m - i)\right) &\subseteq P^{(3)}\left( \left\lceil \frac{m - 1}{3} \right\rceil, 4 + 4(m - 1 - i)\right)\\&\subseteq P^{(3)}\left( \left\lceil \frac{m-1}{3} \right\rceil, 1 + 4(m - 1 - i)\right),
\end{align*}
as desired.\\
\textit{Case 2.} If $m \equiv 1 \pmod{3}$, then $\lceil\frac{m - 1}{3}\rceil = \lceil\frac{m}{3}\rceil - 1$, and we want to show that $P^{(3)}(\lceil \frac{m}{3} \rceil, 2 + 4(m - i)) \subseteq P^{(3)}(\lceil \frac{m-1}{3} \rceil, 4(m - 1 - i))$. Observe that
\begin{align*}
P^{(3)}\left( \left\lceil \frac{m}{3} \right\rceil, 2 + 4(m - i)\right) &\subseteq P^{(3)}\left( \left\lceil \frac{m - 1}{3} \right\rceil + 1, 4 + 2 + 4(m - 1 - i)\right)\\&\subseteq P^{(3)}\left( \left\lceil \frac{m-1}{3} \right\rceil, 4(m - 1 - i)\right),
\end{align*}
as desired.\\
\textit{Case 3.} If $m \equiv 2 \pmod{3}$, then $\lceil\frac{m - 1}{3}\rceil = \lceil\frac{m}{3}\rceil$, and we want to show that $P^{(3)}(\lceil \frac{m}{3} \rceil, 1 + 4(m - i)) \subseteq P^{(3)}(\lceil \frac{m-1}{3} \rceil, 2 + 4(m - 1 - i))$. Observe that
\begin{align*}
P^{(3)}\left( \left\lceil \frac{m}{3} \right\rceil, 1 + 4(m - i)\right) &\subseteq P^{(3)}\left( \left\lceil \frac{m - 1}{3} \right\rceil, 4 + 1 + 4(m - 1 - i)\right)\\
&\subseteq P^{(3)}\left( \left\lceil \frac{m-1}{3} \right\rceil, 2 + 4(m - 1 - i)\right),
\end{align*}
as desired.

Thus, the result holds for $p = 3$. In similar fashion, the result can also be shown to hold for $p = 5, 7$, completing the proof of Theorem~\ref{betterbound}.  This method of proving the polynomial step could be used to create a similar proof for $p = 2$, simplifying the argument in~\cite{jenkins}.

Now Theorem~\ref{congruence} follows easily from Theorem~\ref{betterbound}.

\congruence*

\begin{proof}
The leading coefficient of a polynomial in the set $P^{(p)}(f^\alpha(m),\gamma_p(m, \alpha))$ is divisible by $p^{\gamma_p(m, \alpha)}$. On the other hand, by Theorem~\ref{betterbound}, $U_{(p)}^\alpha \phi^m$ is an element of that set, so every $p^{\alpha}$th coefficient is divisible by $p^{\gamma_p(m, \alpha)}$.
\end{proof}

\section{The Case $p = 13$}

Since $13$ is also a genus zero prime, it is natural to consider whether any congruences hold for $\phi^{(13)}(z)$ or its powers. Computationally, it appears that unless $m \equiv 5 \pmod{13}$, expressing $U_{13} \phi^{(13)}(z)^m$ as a polynomial in $\phi^{(13)}(z)$ gives a coefficient not divisible by $13$. Furthermore, in the case of $m \equiv 5 \pmod{13}$, it appears that there are two coefficients of that polynomial that are exactly divisible by $13$. If we attempt to use the methods in this paper, this prevents us from chaining properly in the polynomial step. There may be some way around this issue, but it would take a different approach. Computationally, for any prime $p$ less than 1000, it appears that the Fourier coefficient of $q^p$ in $\phi^{(13)}(z)$ is divisible by $13$ precisely when $\tau(p)$ is divisible by $13$, but no further congruences of the style given in this paper are immediately apparent.

\bibliographystyle{amsplain}

\end{document}